%% file: main.tex
\documentclass[11pt,oneside]{article}
\usepackage[sort,numbers]{natbib}
\usepackage{graphicx}
\usepackage{multirow}
\usepackage{amsmath,amssymb,amsfonts}
\usepackage{amsthm}
\usepackage{mathrsfs}
\usepackage[title]{appendix}
\usepackage{xcolor}
\usepackage{textcomp}
\usepackage{manyfoot}
\usepackage{booktabs}
\usepackage[shortlabels]{enumitem}
\usepackage{listings}%

\theoremstyle{plain}
\newtheorem{theorem}{Theorem}

\newtheorem{proposition}{Proposition}

\newtheorem{lemma}{Lemma}
\newtheorem{corollary}{Corollary}

\theoremstyle{plain}
\newtheorem{example}{Example}

\theoremstyle{plain}%
\newtheorem{definition}{Definition}%
\newtheorem{condition}{Condition}%
\newtheorem{assumption}{Assumption}%

\usepackage[colorlinks=true,citecolor=blue]{hyperref}

\usepackage[capitalize]{cleveref}

\usepackage{graphicx}

\DeclareMathOperator{\Dist}{Dist}

\crefformat{equation}{(#2#1#3)}
\crefformat{condition}{Condition #2#1#3}
\crefrangeformat{equation}{(#3#1#4)--(#5#2#6)}
\crefrangeformat{condition}{Conditions #3#1#4--#5#2#6}
\crefrangeformat{assumption}{Assumptions #3#1#4--#5#2#6}

\input{macros}

\title{Convergence, Duality and Well-Posedness in Convex Bilevel Optimization}

\usepackage{authblk}
\usepackage[margin=1in]{geometry}

\author[1]{Khanh-Hung Giang-Tran}
\author[2]{Nam Ho-Nguyen}
\author[3]{Fatma K{\i}l{\i}n\c{c}-Karzan}
\author[4]{Lingqing Shen}

\affil[1]{School of Operations Research and Information Engineering, Cornell University}
\affil[2]{Discipline of Business Analytics, The University of Sydney}
\affil[3]{Tepper School of Business, Carnegie Mellon University}

\begin{document}

\maketitle

\begin{abstract}
	We consider the \emph{convex bilevel optimization} problem, also known as \emph{simple bilevel programming}. Strong duality of the convex bilevel optimization problem is not guaranteed, due to the lack of Slater constraint qualification. Furthermore, convergence of algorithms is in general not guaranteed without further conditions due to possibility of super-optimal solutions. We show that strong duality (but not necessarily dual solvability) is exactly equivalent to ensuring correct asymptotic convergence of function values. We provide a simple condition that guarantees strong duality. Unfortunately, we show that this is not sufficient to guarantee convergence to the optimal solution set. We draw connections to Levitin-Polyak well-posedness, and leverage this together with our strong duality equivalence to provide another condition that ensures convergence to the optimal solution set with quantifiable bounds.
\end{abstract}

\section{Introduction}\label{sec:intro}

Consider a constrained convex optimization problem given by
\begin{equation} \label{eq:convex-problem}
    \begin{aligned}
        g^* := \min_{x \in X} g(x),
    \end{aligned}
\end{equation}
where $X$ is a finite-dimensional closed convex set and $g: X \to \bbR$ is a convex function. Here, $g^*$ is the optimal value of \cref{eq:convex-problem}. In general, problem \eqref{eq:convex-problem} might have multiple optimal solutions, and in this case it is natural to
consider a second objective to select the best solution among the optimal solutions with respect to an additional criterion.

Supposing that this is given by a convex function $f: {X}
\to \mathbb{R}$, this leads to the following so-called \emph{convex bilevel optimization problem} (also known as \emph{simple bilevel program}):
\begin{equation}\label{eq:convex-bilevel}
    \begin{aligned}
        f^* := \min_{x\in X_g} f(x), \quad\text {where } X_g:= \argmin_{z\in X} g(z).
    \end{aligned}
\end{equation}
Here, $f^*$ denotes the optimal value of the convex bilevel problem \cref{eq:convex-bilevel}, and $X_g$ is the set of minimizers for \cref{eq:convex-problem}.
Throughout this paper, we assume that problems \eqref{eq:convex-problem} and  \eqref{eq:convex-bilevel} are solvable.
Moreover, we denote an optimal solution of \eqref{eq:convex-bilevel} by $x^*$, satisfying $f(x^*)=f^*$ and $g(x^*)=g^*$.

In contrast to a standard single-level convex optimization problem, the presence of two objective functions in the convex bilevel problem \eqref{eq:convex-bilevel} brings in additional challenges, which we categorize into two groups.
First, we do not have an explicit representation of the optimal set $X_g$ in general. This rules out the possibility of applying projection-based algorithms as well as the conditional gradient method, making both the projection onto $X_g$ and the linear optimization oracle over $X_g$ intractable.

With this first challenge in mind, a stream of literature in this domain has alternatively considered the \emph{value function formulation} of \eqref{eq:convex-bilevel}:
\begin{equation} \label{eq:convex-bilevel-value}
    \min_{x\in X} \set{ f(x) :\ g(x) \leq g^* }.
\end{equation}
This is now a convex optimization problem with a convex functional constraint. However, a critical issue, i.e., our second challenge, is that \eqref{eq:convex-bilevel-value} does not satisfy strict feasibility and hence the standard Slater constraint qualification condition fails. {As a result, it is possible that the Lagrangian dual of \eqref{eq:convex-bilevel-value} is unsolvable, and in such a case \eqref{eq:convex-bilevel-value} will violate a critical assumption exploited in the convergence analysis of most primal-dual methods. One may attempt to enforce Slater's condition by adding a small $\epsilon_g > 0$ to the right-hand side of the constraint, i.e., $g(x) \leq g^* + \epsilon_g$, but this approach provides guarantees only  in terms of $\epsilon_g$-optimality and does not solve the actual problem \eqref{eq:convex-bilevel-value}. 
In fact, it was demonstrated in \citep[Appendix D]{JiangEtAl2023} that such an approach is numerically unstable.

\paragraph{Contributions.} Approaches to solve single-level optimization problems cannot by default be applied to \cref{eq:convex-bilevel} (or equivalently \cref{eq:convex-bilevel-value}), and thus it requires specific algorithms for the bilevel setting, with different analyses. In this short communication, we examine the algorithmic guarantees which exist for various bilevel optimization algorithms, and provide conditions to reassure us of convergence. Algorithmic guarantees for \cref{eq:convex-bilevel} generally ensure that an iteratively computed sequence $\{x_t\}_{t \geq 1} \subset X$ satisfies the following inequalities:
\begin{align}\label{eq:suboptimality-guarantee}
&f(x_t) \leq f^* + r_{f,t}, \quad g(x_t) \leq g^* + r_{g,t}, \quad \textbf{where} \lim_{t \to \infty} r_{f,t} = \lim_{t \to \infty} r_{g,t} = 0.
\end{align}
Unfortunately, as we will demonstrate in \cref{sec:conditions}, this is \emph{not} sufficient to ensure that $\{x_t\}_{t \geq 1}$ converges to a solution of \cref{eq:convex-bilevel}. Our study aims to analyze this carefully, and provide conditions which ensure convergence. It turns out that this can be studied by examining the regularity of the Lagrange dual of \cref{eq:convex-bilevel-value}. As mentioned above, \cref{eq:convex-bilevel-value} does not satisfy Slater's condition is by definition. That said, we can still analyze its Lagrange dual, and examine conditions for strong duality to hold. Here, we must differentiate between two cases: (i) one where the optimal dual value is equal to $f^*$; and a stronger case where additionally there is an optimal dual multiplier $\lambda^*$. We will show that the first case is in fact equivalent to ensuring that sequences satisfying \cref{eq:suboptimality-guarantee} will in fact converge to an optimal solution of \cref{eq:convex-bilevel}. This uncovers a useful hidden regularity structure in \cref{eq:convex-bilevel} that, to the best of our knowledge, has not been elaborated on previously. We conclude the paper by presenting two conditions to ensure that strong duality holds, and examine how these conditions have been implicitly present in existing algorithmic work.

\paragraph{Related work.} Much of the work on convex bilevel optimization involves the design and analysis of algorithms to solve \cref{eq:convex-bilevel}. There are several different schemes, including \emph{iterative regularization} \citep{Cabot2005,Solodov2006,SolodovMikhail2007,HelouSimões2017,Malitsky2017,DempeEtAl2019,AminiEtAl2019,KaushikYousefian2021,ShenLingqing2023,GiangTranEtAL2024,MerchavEtAl2024,SamadiEtAl2024,ChenEtAl2024,ShternTaiwo2025}, \emph{sublevel set methods} \citep{BeckSabach2014,JiangEtAl2023,CaoEtAl2023,DoronShtern2022,CaoEtAl2024,WangEtAl2024} and \emph{sequential averaging} \citep{SabachEtAl2017,ShehuEtAl2021,MerchavSabach2023}.

Our work focuses on the duality theory of \cref{eq:convex-bilevel}, namely the Lagrange dual of its value function formulation \cref{eq:convex-bilevel-value}. \citet{FriedlanderEtAl2008} show that correctness of the exact regularization technique is equivalent to the existence of an optimal dual multiplier of \cref{eq:convex-bilevel-value}, and they demonstrate a number of special cases where exact regularization holds, e.g., conic programs and exact penalty methods for convex functional constraints. \citet{DempeEtAl2010,FrankeEtAl2018} provide conditions to guarantee existence of an optimal dual multiplier by using more general constraint qualifications than Slater. In contrast, our work focuses on when strong duality holds \emph{regardless} of whether the dual is solvable or not, and in particular we highlight that this case has important algorithmic implications for \cref{eq:convex-bilevel}. Consequently, we are able to provide weaker conditions for strong duality, and thus convergence assurance, than the ones above that guarantee that an optimal dual multiplier exists. \citet{FriedlanderEtAl2008} observe that, under boundedness assumptions, strong duality holds (without solvability), which is a simple consequence of Sion's minimax theorem, though they did not explore this further.

Finally, we note that there is a broader literature on general bilevel optimization problems with two sets of variables, which are non-convex; see \citet{Dempe2002book} for details. This literature also examines optimality conditions via Lagrange multipliers, though the theory is different to the simpler problem \cref{eq:convex-bilevel} that we consider.

\paragraph{Notation and basic assumption.}
We let $\R$ and $\N$ denote the real and natural numbers, respectively. Given any positive integer $T$, let $[T] := \{1, 2, \dots, T\}$. For any $a\in\R$, we let $[a]_+:=\max\{0,a\}$. We define $\{a_t\}_{t\in[T]}$ to be the sequence of elements $(a_1, a_2, \dots, a_T)$. Let $\|\cdot\|_2$ be the Euclidean norm. For a convex function $f : X \to \R$, we denote the gradient at $x \in X$ by $\nabla f (x)$. We will assume that \cref{eq:convex-bilevel} satisfies the following properties throughout the paper.
\begin{assumption}\label{ass:basic}
The set $X$ is closed and convex, and belongs to some finite-dimensional Euclidean space. The functions $f$ and $g$ are convex and real-valued on all of $X$. Furthermore, the minimization problem \cref{eq:convex-problem} is solvable, i.e., $g^* < \infty$ and $X_g$ defined in \cref{eq:convex-bilevel} is non-empty and closed. Finally, the bilevel optimization problem \cref{eq:convex-bilevel} is solvable, and we denote the set of optimal solutions to \cref{eq:convex-bilevel} as
\[ X^* := \argmin_{x \in X_g} f(x) = \{ x \in X_g : f(x) = f^* \}, \]
which is non-empty and closed.
\epr
\end{assumption}

\section{Super-optimality and Lagrange duality}\label{sec:lagrange-dual}

The central issue in solving \cref{eq:convex-bilevel} is, in practice, we very rarely encounter a feasible solution $\Bar{x} \in X_g$. Instead, algorithms will generate a point $x_t \in X$ with $g(x_t) > g^*$. Since $X_g \subset \left\{ x \in X : g(x) \leq g(x_t) \right\}$, the point $x_t$ may be \emph{super}-optimal for \cref{eq:convex-bilevel}, meaning that $f(x_t) < f^*$. As iterative algorithms progress, the accuracy of $x_t$ (hopefully) improves, i.e., it is shown that $g(x_t) \leq g^* + r_{g,t}$ where $r_{g,t} \geq 0$ are error terms such that $r_{g,t} \to 0$. Since $g(x_t) \geq g^*$, this implies that $\lim_{t \to \infty} g(x_t) = g^*$. For the outer-level objective $f$, it is often similarly shown that $f(x_t) \leq f^* + r_{f,t}$, where $r_{f,t} \geq 0$ are error terms such that $r_{f,t} \to 0$. Since we might have $f(x_t) < f^*$, these inequalities are only sufficient to deduce
\begin{equation} \label{eq:weak-asymptotic}
\limsup_{t \to \infty} f(x_t)\leq f^*, \quad \lim_{t \to \infty} g(x_t) = g^*,
\end{equation}
but they do \emph{not} guarantee $f(x_t) \to f^*$. Indeed, we may have a subsequence $\{x_{t_k}\}_{k \geq 1}$ for which $\underline{f} := \lim_{k \to \infty} f(x_{t_k}) < f^*$. On the other hand, it is straightforward to show that any limit point of $\{x_t\}_{t \geq 1}$ is an optimal solution of \cref{eq:convex-bilevel}. That said, it is difficult to infer in an online fashion whether the current $f(x_t)$ is close to $f^*$ or $\underline{f}$. In other words, we have no assurance of the approximation quality of $f(x_t)$.

We now present an example that demonstrates how examining $f(x_t)$ without further assurances can indeed be misleading for estimating $f^*$.
\begin{example}\label{ex:counter-example}
Suppose that
\begin{equation}\label{eq:counter}
	\begin{aligned}
		f(x)&:=x_2^2-2x_2 + [100-x_1]_+^2,\quad
		g(x):= \frac{x_2^2}{x_1}, \quad X := \set{x \in \R^2:\, x_1 \geq 1,~ 0 \leq x_2 \leq 1}.
	\end{aligned}
\end{equation}
Clearly, $f,g,X$ are convex. Furthermore, we have $X_g = \{(s,0) \in \R^2 :\, s \geq 1\}$, $g^* = 0$, $X^* = \{ (s,0) \in \bbR^2 : s \geq 100 \}$ and $f^* = 0$. However, consider now the sequence $\{x_t\}_{t \geq 1} \subset X$ given by $x_t := (t,1)$ for $t \geq 1$. Then $g(x_t) = 1/t \to 0 = g^*$, but $f(x_t) = [100-t]_+^2 - 1 \to -1 < f^*$. Thus, $\{x_t\}_{t \geq 1}$ satisfies \cref{eq:weak-asymptotic} but $f(x_t)$ is not a good estimate for $f^*$.
\epr
\end{example}

To guard against situations like in \cref{ex:counter-example}, and recognizing the numerous algorithmic works which provide guarantees of the form \cref{eq:weak-asymptotic}, we look for conditions which ensure that whenever \cref{eq:weak-asymptotic} holds, the following also holds:
\begin{equation} \label{eq:strong-asymptotic}
	\lim_{t \to \infty}f(x_t) = f^*, \quad \lim_{t \to \infty} g(x_t) = g^*.
\end{equation}
Another possible condition we consider is the that the distance to the optimal solution set of \cref{eq:convex-bilevel} decreases to $0$:
\begin{equation} \label{eq:strong-solution}
	\lim_{t \to \infty} \Dist(x_t,X^*) = 0, \quad \text{where } \Dist(x,X^*) := \inf_{z \in X^*} \|x-z\|_2.
\end{equation}

Interestingly, we uncover a relationship between Lagrange duality and \cref{eq:strong-asymptotic}. The Lagrange dual of \cref{eq:convex-bilevel} is obtained via \cref{eq:convex-bilevel-value}, and is given by
\begin{equation}\label{eq:lagrange-dual}
d^* := \sup_{\lambda \geq 0} q(\lambda), \quad \text{where} \quad q(\lambda) := \inf_{x \in X} \left\{ f(x) + \lambda (g(x) - g^*) \right\}.
\end{equation}
Here, $\lambda \geq 0$ is the dual variable, $q(\cdot)$ is the (concave) dual objective function, and $d^*$ is the optimal dual value. Since $g(x) \geq g^*$ for any $x \in X$, we can deduce that $q(\lambda)$ is non-decreasing in $\lambda$. This means that $d^* = \sup_{\lambda \geq 0} q(\lambda) = \lim_{\lambda \to \infty} q(\lambda)$. Recalling that $f^*$ is the optimal value of \cref{eq:convex-bilevel-value}, by weak duality we always have $d^* \leq f^*$. We will use the following definition of strong duality.
\begin{definition}\label{def:strong-duality}
We say that \emph{strong duality} holds for \cref{eq:convex-bilevel} when $d^* = f^*$.
\epr
\end{definition}
Note that \cref{def:strong-duality} does \emph{not} guarantee that the dual is solvable, but only that $q(\lambda) \to f^*$ as $\lambda \to \infty$. Therefore, ensuring that strong duality holds requires much weaker conditions than ensuring that the dual is solvable, as was studied in previous works \citep{FriedlanderEtAl2008,DempeEtAl2010,FrankeEtAl2018}.

Our main result is that ensuring convergence \cref{eq:strong-asymptotic} is \emph{exactly equivalent} to strong duality holding for \cref{eq:convex-bilevel}.
\begin{theorem}\label{thm:duality-equivalence}
Considering \cref{eq:convex-bilevel} when $f,g,X$ are convex, the following are equivalent:
\begin{enumerate}
    \item Any sequence $\{x_t\}_{t \geq 1} \subset X$ that satisfies \cref{eq:weak-asymptotic} also satiesfies \cref{eq:strong-asymptotic}.

    \item Strong duality (\cref{def:strong-duality}) holds for \cref{eq:convex-bilevel}.
\end{enumerate}
\end{theorem}
We will use the following lemma in the proof of \cref{thm:duality-equivalence}, as well as subsequent results.
\begin{lemma}\label{lemma:dual-value-sequence}
There exists a sequence $\{x_t\}_{t \geq 1} \subset X$ such that $g(x_t) \to g^*$ and $\limsup_{t \to \infty} f(x_t) \leq d^*$.
\end{lemma}
\begin{proof}[Proof of \cref{lemma:dual-value-sequence}]
We consider two cases: $d^* > -\infty$ and $d^* = \infty$. In the first case, there must exist some $\lambda_0 \geq 0$ such that $q(\lambda_0) = \inf_{x \in X} \left\{ f(x) + \lambda_0 (g(x) - g^*) \right\} > -\infty$, and consequently, $q(\lambda) \geq q(\lambda_0) > -\infty$ for any $\lambda \geq \lambda_0$. For each $t \geq 1$, choose $x_t \in X$ to be a point that satisfies $f(x_t) + (t+\lambda_0) (g(x_t) - g^*) \leq q(t+\lambda_0) + 1/t$. We also have $f(x_t) + \lambda_0 (g(x_t) - g^*) \geq q(\lambda_0)$, so combining with $f(x_t) + (t+\lambda_0) (g(x_t) - g^*) \leq q(t+\lambda_0)$ we get
\[ g(x_t) - g^* \leq \frac{1}{t} \left( q(t+\lambda_0) + 1/t - q(\lambda_0) \right). \]
This implies that $g(x_t) \to g^*$. Now observe that
\begin{align*}
f(x_t) + \lambda_0 (g(x_t) - g^*) &\leq f(x_t) + (t+\lambda_0) (g(x_t) - g^*)\\
&\leq q(t+\lambda_0) + 1/t.
\end{align*}
Taking limit superior of all terms in the inequality, we have
\[ \limsup_{t \to \infty} f(x_t) = \limsup_{t \to \infty} \left( f(x_t) + \lambda_0 (g(x_t) - g^*) \right) \leq \limsup_{t \to \infty} \left( q(t+\lambda_0) + 1/t \right) = d^*. \]
The first equality follows because $g(x_t) \to g^*$, and the last equality follows since $\lim_{t \to \infty} = d^*$.

Now we consider the second case when $d^* = -\infty$. Then for $t \geq 1$, let $z_t \in X$ be a point such that $f(z_t) + t(g(z_t) - g^*) \leq -t$, and define $\theta_t = \frac{1}{\sqrt{t} \max\{1, g(z_t) - g^*\}} \in [0,1]$. Let $x^* \in X^*$ be a solution to \cref{eq:convex-bilevel}, and define $x_t = \theta_t z_t + (1-\theta_t) x^* \in X$. Notice that
$g(x_t) \leq \theta_t g(z_t) + (1-\theta_t) g(x^*) = g^* + \theta_t (g(z_t) - g^*) = g^* + \frac{g(z_t) - g^*}{\sqrt{t} \max\{1,g(z_t)-g^*\}} \leq g^* + 1/\sqrt{t} \to g^*$. Therefore $\lim_{t \to \infty} g(x_t) = g^*$. On the other hand, $f(x_t) \leq \theta_t f(z_t) + (1-\theta_t) f(x^*)$, and by definition of $z_t$, we have $f(z_t) \leq -t(1 + g(z_t) - g^*)$, so $\theta_t f(z_t) \leq -\frac{\sqrt{t} (1+g(z_t) - g^*)}{\max\{1,g(z_t)-g^*\}}$. But this implies $\theta_t f(z_t) \to -\infty$ as $t \to \infty$, therefore $\limsup_{t \to \infty} f(x_t) = -\infty = d^*$.
\end{proof}

\begin{proof}[Proof of \cref{thm:duality-equivalence}]
Suppose that strong duality holds. Let $\{x_t\}_{t \geq 1}$ be a sequence satisfying \cref{eq:weak-asymptotic}. Then given any $\epsilon > 0$, there exists $\lambda_\epsilon \in [0,\infty)$ such that
\[
\inf_{x \in X} \left\{f(x)+\lambda_\epsilon(g(x)-g^*)\right\} \geq f^* -\epsilon.
\]
This implies that for any $t \geq 1$ we have
$f(x_t) \geq f^* -\lambda_\epsilon(g(z_t)-g^*)-\epsilon$. By taking limit inferior as $t\to\infty$ on both sides, and recognizing that $\lim_{t \to \infty} g(x_t) = g^*$, we obtain $\liminf_{t \to \infty} f(x_t) \geq f^* -\epsilon$. Since $\epsilon >0$ is chosen arbitrarily, $\liminf_{t \to \infty} f(z_t) \geq f^*$, which implies \eqref{eq:strong-asymptotic}.

Suppose now that strong duality does not hold, so $d^* = \sup_{\lambda \geq 0} q(\lambda) < f^*$. \cref{lemma:dual-value-sequence} states that there exists a sequence $\{x_t\}_{t \geq 1}$ such that $g(x_t) \to g^*$ and $\limsup_{t \to \infty} f(x_t) \leq d^*$. Since $d^* < f^*$, \cref{eq:weak-asymptotic} holds for $\{x_t\}_{t \geq 1}$, but not \cref{eq:strong-asymptotic}.
\end{proof}

\section{Conditions for strong duality and convergence}\label{sec:conditions}

We now provide two conditions to ensure that strong duality holds for \cref{eq:convex-bilevel}, thereby (equivalently) ensuring that any algorithm which guarantees \cref{eq:weak-asymptotic} has assurance that \cref{eq:strong-asymptotic} holds also.
\begin{condition}\label{cond:dist-g}
If $\{x_t\}_{t \geq 1} \subset X$ is a sequence such that $g(x_t) \to g^*$, we have $\Dist(x_t,X_g) \to 0$.
\epr
\end{condition}
\begin{condition}\label{cond:compact}
The set of optimal solutions $X^*$ for \cref{eq:convex-bilevel} is compact.
\epr
\end{condition}

\cref{cond:dist-g} generalizes, for example, error bound-type conditions on $g$, where $\tau \Dist(x,X_g)^r \leq g(x) - g^*$. Error bounds are used in several works \citep[e.g.,][]{JiangEtAl2023,CaoEtAl2024,MerchavEtAl2024} to enable stronger guarantees for algorithms solving \cref{eq:convex-bilevel}.

\cref{cond:compact} can be shown to be equivalent to coercivity of the functions $h = \max\{f,g\}$ and $\bar{h} = \max\{f-f^*,g-g^*\}$. \cref{cond:compact}, or more restrictive versions of it, has appeared in various works. \citet{Cabot2005,Solodov2006,SolodovMikhail2007,Malitsky2017,DempeEtAl2019} assume exactly \cref{cond:compact}. \citet{HelouSimões2017,AminiEtAl2019,KaushikYousefian2021,ShenLingqing2023,GiangTranEtAL2024,JiangEtAl2023,CaoEtAl2024} employ the assumption that $X$ is bounded to prove guarantees for their algorithms. This immediately implies \cref{cond:compact} holds. Finally, \citet{BeckSabach2014,SabachEtAl2017,ShehuEtAl2021,DoronShtern2022,MerchavSabach2023} use the assumption that $f$ is coercive, which also implies \cref{cond:compact}. This is commonly motivated by learning problems where $g$ models the prediction error and $f$ is a regularization function, e.g., the $\ell_1$- or $\ell_2$-norm.

\cref{cond:dist-g} guarantees that strong duality holds, and hence \cref{eq:weak-asymptotic} implies \cref{eq:strong-asymptotic}.

\begin{proposition}\label{prop:dist-g-strong-duality}
Suppose that \cref{cond:dist-g} holds for \cref{eq:convex-bilevel}. Then strong duality holds, and, consequently, any sequence satisfying \cref{eq:weak-asymptotic} also satisfies \cref{eq:strong-asymptotic}.
\end{proposition}
\begin{proof}[Proof of \cref{prop:dist-g-strong-duality}]
Suppose that $\{x_t\}_{t \geq 1}$ is a sequence satisfying the premise of \cref{lemma:dual-value-sequence}: $g(x_t) \to g^*$ and $\limsup_{t \to \infty} f(x_t) \leq d^*$. Let $x^* \in X^*$ be an optimal solution to \cref{eq:convex-bilevel}, and define the sequence $\{y_t\}_{t \geq 1} \subset X_g$ by $y_t := \argmin_{y \in X_g} \|x_t - y\|_2$. (Note that the minimizer is unique since the $\ell_2$-norm is strictly convex.) By convexity of $f$, we have
\begin{align*}
f(x_t) - f^* &= f(x_t) - f(x^*)\\
&\geq \grad f(x^*)^\top (x_t - x^*)\\
&= \grad f(x^*)^\top (x_t - y_t) + \grad f(x^*)^\top (y_t - x^*)\\
&\geq -\|\grad f(x^*)\|_2 \|x_t - y_t\|_2 + \grad f(x^*)^\top (y_t - x^*)\\
&\geq -\|\grad f(x^*)\|_2 \|x_t - y_t\|_2.
\end{align*}
The second inequality follows from the Cauchy-Schwarz inequality, while the third follows since $\grad f(x^*)^\top (y_t - x^*) \geq 0$ by the optimality of $x^*$ in minimizing $f$ over $X_g$, and the fact that $y_t \in X_g$. By \cref{cond:dist-g}, we have $\|x_t - y_t\|_2 = \Dist(x_t,X_g) \to 0$ since $g(x_t) \to g^*$. Therefore, the limit inferior of both sides is
\[ \liminf_{t \to \infty} f(x_t) - f^* \geq 0. \]
In summary, we have
\[ f^* \leq \liminf_{t \to \infty} f(x_t) \leq \limsup_{t \to \infty} f(x_t) \leq d^*, \]
so strong duality holds.
\end{proof}
Unfortunately, \cref{cond:dist-g} on its own is not enough to guarantee that \cref{eq:strong-solution} holds.
\begin{example}\label{ex:counter-dist-g}
Define $X = \{(x_1,x_2,x_3) : x_1 \geq 1, 0 \leq x_2 \leq 1\} \subset \bbR^3$, and $g(x) = x_3^2$. Then $g^* = 0$ and $X_g = \{(x_1,x_2,x_3) : x_1 \geq 1, 0 \leq x_2 \leq 1, x_3 = 0\}$. Notice that for any $x \in X$, $\Dist(x,X_g) = |x_3| = \sqrt{g(x)}$, therefore \cref{cond:dist-g} is satisfied. Now define $f(x) = x_2^2/x_1$, which is convex. Then $f^* = 0$ and $X^* = \{(x_1,x_2,x_3) : x_1 \geq 1, x_2 = 0, x_3 = 0\}$. Now consider the sequence $\{x_t\}_{t \geq 1}$ defined by $x_t = (t,1,0) \in X_g$. Then $f(x_t) = 1/t \to f^*$, $g(x_t) = 0 = g^*$, $\Dist(x_t,X_g) = 0$ but $\Dist(x_t,X^*) = 1$ for all $t$. Therefore \cref{eq:strong-solution} is not satisfied.
\epr
\end{example}

\cref{cond:compact} is more powerful than \cref{cond:dist-g}, and in fact has strong connections to the study of well-posedness, which examines properties of approximation schemes for solutions of optimization (and related) problems. 
\begin{definition}[{\citet{HuangYang2006}}]\label{def:Levitin-Polyak-wellposed}
We say that \cref{eq:convex-bilevel} satisfies:
\begin{itemize}
    \item \emph{Levitin-Polyak (LP) well-posedness} if, for any sequence $\{x_t\}_{t \geq 1} \subset X$ which satisfies $\Dist(x,X_g) \to 0$ and $f(x_t) \to f^*$, there exists a subsequence $\{x_{t_k}\}_{k \geq 1} \subset \{x_t\}_{t \geq 1}$ and $\bar{x} \in X^*$ such that $x_{t_k} \to \bar{x}$.

    \item \emph{generalized LP well-posedness} if, for any sequence that satisfies \cref{eq:strong-asymptotic}, there exists a subsequence $\{x_{t_k}\}_{k \geq 1} \subset \{x_t\}_{t \geq 1}$ and $\bar{x} \in X^*$ such that $x_{t_k} \to \bar{x}$.

    \item \emph{strong generalized LP well-posedness} if, for any sequence that satisfies \cref{eq:weak-asymptotic}, there exists a subsequence $\{x_{t_k}\}_{k \geq 1} \subset \{x_t\}_{t \geq 1}$ and $\bar{x} \in X^*$ such that $x_{t_k} \to \bar{x}$.
\end{itemize}
\epr
\end{definition}
It is immediately clear that \emph{strong} generalized LP well-posedness implies generalized LP well-posedness. We will show that every sense of well-posedness in \cref{def:Levitin-Polyak-wellposed} implies \cref{cond:compact}.
It turns out that strong generalized LP well-posedness is equivalent to \cref{cond:compact} in finite dimensions (which is our setting), which follows from a result of \citet{HuangYang2006}. We will additionally show that \cref{cond:compact} ensures that strong duality holds, and thus \cref{thm:duality-equivalence} implies that generalized LP is equivalent to strong generalized LP well-posedness for \cref{eq:convex-bilevel}.

\begin{proposition}\label{prop:compact-equivalence}
The following are equivalent for \cref{eq:convex-bilevel}:
\begin{enumerate}[label=(\alph*)]
    \item \cref{cond:compact} holds.
    \item \cref{eq:convex-bilevel} is generalized LP well-posed.
    \item \cref{eq:convex-bilevel} is strongly generalized LP well-posed.
\end{enumerate}
Furthermore, if \cref{eq:convex-bilevel} is LP well-posed, then \cref{cond:compact} holds.
\end{proposition}
\begin{proof}[Proof of \cref{prop:compact-equivalence}]
First, note that if $X^*$ is not compact, then it must be unbounded, since we assume that $X^*$ is closed. Take a sequence $\{x_t\}_{t \geq 1} \subset X^*$ such that $\|x_t\|_2 \to \infty$. Then this sequence satisfies all three premises of \cref{def:Levitin-Polyak-wellposed}, however there is no convergent subsequence. Therefore all three senses of well-posedness imply compactness of $X^*$, i.e., \cref{cond:compact}. \citet[Theorem 2.4]{HuangYang2006} states that when $X$ is finite-dimensional, which we assume, then strong generalized LP well-posedness is equivalent to \cref{cond:compact}.

To show that strong generalized LP is equivalent to generalized LP, we will first show that under \cref{cond:compact}, strong duality holds. First, denote
\[ X_1 := \left\{ x \in X : f(x) - f^* \leq 1, g(x) - g^* \leq 1 \right\}. \]
Note that $X_1$ is the $1$-sublevel set of the function $\bar{h} = \max\{f-f^*,g-g^*\}$, and $X^*$ is the $0$-sublevel set. Since $X^*$ is bounded, all sublevel sets of $\bar{h}$ are bounded, hence $X_1$ is also bounded.
Let $\{x_t\}_{t \geq 1} \subset X$ be a sequence satisfying the premise of \cref{lemma:dual-value-sequence}, i.e., $g(x_t) \to g^*$ and $\limsup_{t \to \infty} f(x_t) \leq d^*$. By weak duality $d^* \leq f^*$, therefore $f(x_t) \leq f^* + 1$ eventually, and also $g(x_t) \leq g^* + 1$ eventually as well. This means that only finitely many $x_t \not\in X_1$, and since $X_1$ is bounded, the whole sequence is bounded. Take a convergent subsequence $\{x_{t_k}\}_{k \geq 1}$ such that $x_{t_k} \to \bar{x}$. Since $X$ is closed, $\bar{x} \in X$. Furthermore, since $g$ is convex, it is lower semicontinuous, so $g^* = \liminf_{k \to \infty} g(x_{t_k}) \geq g(\bar{x}) \geq g^*$, hence $\bar{x} \in X_g$. But this implies that $f(\bar{x}) \geq f^*$ by definition of $f^*$. Therefore $f^* \leq \limsup_{t \to \infty} f(x_t) \leq d^*$, and thus strong duality holds.

Now, assume that \cref{eq:convex-bilevel} is generalized LP well-posed. Then $X^*$ must be compact. Take a sequence $\{x_t\}_{t \geq 1} \subset X$ satisfying \cref{eq:weak-asymptotic}. Since compactness of $X^*$ implies strong duality, \cref{thm:duality-equivalence}, we know that $\{x_t\}_{t \geq 1}$ also satisfies \cref{eq:strong-asymptotic}. Therefore by generalized LP well-posedness, there exists $\bar{x} \in X^*$ and a subsequence for which $x_{t_k} \to \bar{x}$, hence \cref{eq:convex-bilevel} is also strongly generalized LP well-posed. Since a strongly generalized LP well-posed problem is also generalized LP well-posed, this shows the equivalence.
\end{proof}

We now use this equivalence to provide a quantitative version of \cref{eq:strong-solution}.

\begin{corollary}\label{cor:quantified}
Suppose that \cref{cond:compact} holds for \cref{eq:convex-bilevel}. Then there exists a function $c:\bbR_+^2 \to \bbR_+$ such that $c(0,0) = 0$ and for any $x \in X$,
\[ |f(x) - f^*| \geq c\left( \Dist(x,X^*), g(x) - g^* \right). \]
Furthermore, the function $c$ satisfies that whenever $\beta_t \to 0$, $c(\alpha_t,\beta_t) \to 0$, then $\alpha_t \to 0$. Therefore, under \cref{cond:compact} any sequence $\{x_t\}_{t \geq 1}$ satisfying \cref{eq:weak-asymptotic} also satisfies \cref{eq:strong-solution}.
\end{corollary}
\begin{proof}[Proof of \cref{cor:quantified}]
\cref{prop:compact-equivalence} states that \cref{cond:compact} is equivalent to generalized LP well-posedness. \citet[Theorem 2.2]{HuangYang2006} then ensures that whenever \cref{eq:convex-bilevel} is generalized LP well-posed, the function $c$ that satisfies the required properties is
\[ c(\alpha,\beta) := \inf_x \left\{ |f(x) - f^*| : x \in X, \Dist(x,X^*) = \alpha, g(x) - g^* = \beta \right\}. \]
\end{proof}

\section{Conclusion}\label{sec:conclusion}

We studied convergence guarantees for the convex bilevel problem \eqref{eq:convex-bilevel}. Through \cref{ex:counter-example} we showed that typical bounds \eqref{eq:weak-asymptotic} are insufficient: they can hold while neither the objective values nor the iterates approach a bilevel solution in a meaningful sense. Our main result establishes an exact remedy: \emph{strong duality} for the value-function formulation \eqref{eq:convex-bilevel-value} is \emph{equivalent} to the property that every sequence satisfying \eqref{eq:weak-asymptotic} necessarily enjoys objective convergence \eqref{eq:strong-asymptotic} (\cref{thm:duality-equivalence}).

We then identified two simple conditions ensuring this regularity. First, the proximity assumption \cref{cond:dist-g} guarantees strong duality and hence \eqref{eq:strong-asymptotic}; however, as \cref{ex:counter-dist-g} demonstrates, it does not ensure convergence to the optimal solution set \eqref{eq:strong-solution}. Second, the compactness assumption \cref{cond:compact} is stronger but widely utilized in previous work. It yields strong duality of \cref{eq:convex-bilevel} and solution-set convergence (\cref{prop:compact-equivalence}) together with quantifiable bounds. These findings provide algorithm-agnostic criteria under which iterate sequences for \cref{eq:convex-bilevel} have meaningful limits.

Interestingly, \cref{prop:compact-equivalence} and \cref{ex:counter-dist-g} shows that \cref{eq:convex-bilevel} may still be ill-posed (in the sense of \cref{def:Levitin-Polyak-wellposed}) even when strong duality holds. Future work includes investigating the implications of the stronger notion of dual solvability on approximating sequences for \cref{eq:convex-bilevel}. Furthermore, obtaining exact forms of the quantified bound in \cref{cor:quantified} for important bilevel problems is of interest.

\input{main.bbl}
\end{document}

%% file: macros.tex
\definecolor{gold}{rgb}{0.85,0.65,0}

\newcommand{\epr}{\hfill\hbox{\hskip 4pt \vrule width 5pt height 5pt depth 0pt}\vspace{0.0cm}\par}
 
\newcommand{\grad}{\ensuremath{\nabla}}

\newcommand{\set}[1]{\left\{#1\right\}}

\def\N{{\mathbb{N}}}

\def\R{{\mathbb{R}}}
\def\bbR{{\mathbb{R}}}

\DeclareMathOperator*{\argmin}{arg\,min}

